\documentclass{amsart}
\usepackage[utf8]{inputenc}
\usepackage{amsfonts,latexsym,amssymb,amsmath,amscd,amsthm}
\usepackage{graphicx,color}
\usepackage{enumerate,url}

\newtheorem{teorema}{Theorem}
\newtheorem{propo}{Proposition}

\newtheorem*{teorema1}{Theorem 1}

\theoremstyle{definition}

\theoremstyle{remark}

\newtheorem{ejemplo}{Example}

\newtheoremstyle{teoremacita}
{3pt}
{3pt}
{\itshape}
{}
{\bfseries}
{}
{ }
{\thmname{#1}\thmnumber{ #2'}\thmnote{ #3}.}
\theoremstyle{teoremacita} \newtheorem*{teor*}{}

\newcommand{\be}{\begin{enumerate}}
\newcommand{\ee}{\end{enumerate}}

\newcommand{\bi}{\begin{itemize}}
\newcommand{\ei}{\end{itemize}}

\newcommand{\dis}{\displaystyle}

\newcommand{\norm}[1]{\left | #1 \right |} 
\newcommand{\norma}[1]{\left |\left| #1 \right|\right |} 
\newcommand{\lista}[2]{{#1}_1,{#1}_2,\ldots ,{#1}_{#2}} 
\newcommand{\pa}[1]{\frac{\partial}{\partial #1}}
\newcommand{\der}[2]{\frac{\partial #1}{\partial #2}}

\newcommand{\RR}{{\mathbb R}}
\newcommand{\ZZ}{{\mathbb Z}}
\newcommand{\NN}{{\mathbb N}}
\newcommand{\QQ}{{\mathbb Q}}
\newcommand{\CC}{{\mathbb C}}

\newcommand{\XX}{{\mathcal X}}
\newcommand{\YY}{{\mathcal Y}}
\newcommand{\BB}{{\mathcal B}}

\newcommand{\OO}{{\mathcal O}}
\newcommand{\FF}{{\mathcal F}}

\newcommand{\bo}{{\mathbf 0}}

\newcommand{\cen}[1]{(\CC^{#1},\bo)}

\DeclareMathOperator{\diff}{Diff}

\begin{document}
\title{Generalized Poincaré-Dulac Singularities of Holomorphic Foliations}

\author{Percy Fern\'{a}ndez-S\'{a}nchez}
\address[Percy Fern\'{a}ndez Sánchez]{Dpto. Ciencias - Secci\'{o}n Matem\'{a}ticas, Pontificia Universidad Cat\'{o}lica del Per\'{u}, Av. Universitaria 1801,
San Miguel, Lima 32, Peru}
\email{pefernan@pucp.edu.pe}

\author{Jorge Mozo-Fern\'{a}ndez}
\address[Jorge Mozo Fern\'{a}ndez]{Dpto. \'{A}lgebra, An\'{a}lisis Matem\'{a}tico, Geometr\'{\i}a y Topolog\'{\i}a \\
Facultad de Ciencias, Universidad de Valladolid \\
Campus Miguel Delibes\\
Paseo de Bel\'{e}n, 7\\
47011 Valladolid - Spain}
\email{jorge.mozo@uva.es}





\date{\today}

\begin{abstract}
In this paper, we study the analytic classification of a class of nilpotent singularities of holomorphic foliations in $(\CC^2,0)$, those exhibiting a Poincaré-Dulac type singularity in their reduction process. This analytic classification is based in the holonomy of a certain component of the exceptional divisor. Finally, as a consequence, we show that these singularities exhibit a formal analytic rigidity.

\end{abstract}

\maketitle

\subsection*{Keywords} Holomorphic foliations, singularities, Poincaré-Dulac, holonomy.

\subsection*{Mathematics Subject Classification} Primary 32S65, Secondary 37C15.

\section{Introduction}

This paper is a contribution to the study of the analytic classification of germs of holomorphic foliations of codimension one defined over an ambient space of dimension two. Such a foliation may be defined, either by a holomorphic 1-form
\begin{equation} \label{forma}
\omega = a(x,y)dx + b(x,y) dy,
\end{equation}
or by its \textit{dual} vector field
\begin{equation} \label{campo}
\mathcal{X}= b(x,y) \pa{x} - a(x,y) \pa{y},
\end{equation}
where $a(x,y)$, $b(x,y)\in \CC \left\{ x,y \right\}$. The origin is a singular point if $a(0,0)= b(0,0)=0$, fact that will be assumed in the sequel.

A problem which has attracted the attention of many mathematicians from long time ago is the classification of such objects, under the action of a convenient group of transformations. Two vector fields
$$
\XX_i= b_i (x,y) \pa{x} - a_i (x,y) \pa{y},  \ i =1,2
$$
as in \eqref{campo} are analytically equivalent if there exists a germ of biholomorphism $\Phi : \cen{2} \longrightarrow \cen{2}$ carrying orbits into orbits, in the sense that 
$$
\Phi_{\ast} \XX_1 = \XX_2.
$$
The previous expression can be written formally in terms of the power series expansions of $\Phi$, $a_i(x,y)$, and $b_i(x,y)$. Then, formal equivalence can also be defined when $\Phi \in \CC [[x,y]]^{2}$ defines an invertible transformation.

Alternatively, two foliations defined by 1-forms $\omega_1$, $\omega_2$ as in \eqref{forma} are analytically equivalent if $\Phi^{\ast} \omega_2 = \omega_1$, with $\Phi$ as above. Formal equivalence in this sense may be accordingly defined.

Both notions of equivalence are substantially different. When discussing the equivalence of vector fields, we take into account the (complex) time parameterizing the solutions, whereas when considering 1-forms, we only take into account the leaves, disregarding the time. Therefore, two analytically equivalent vector fields define analytically equivalent foliations, but the converse is not true: if $U(x,y)$ is a unit, both $\omega$ and $U(x,y) \omega$ define the same foliation, but the vector fields $\XX$ and $U(x,y)\XX$ are not necessarily analytically equivalent in the sense of vector field equivalence. Some authors distinguish these two notions by referring to vector fields as either equivalent or orbitally equivalent \cite{Strozyna,SZ,SZ2}.

In this work, we are interested in the analytic equivalence of foliations, i.e., what is referred to as orbital analytic equivalence in the aforementioned papers. Consequently, throughout the paper, we will primarily work with holomorphic 1-forms instead of vector fields.

The framework of this paper is the search for criteria to determine when two formally equivalent germs of foliations are also analytically equivalent, and to study the formal-analytic moduli of these objects. It is convenient to recall here that any such germ admits a reduction of singularities: consider the map $\pi : (E,D) \longrightarrow \mathbb{C}^2$ which consists of blowing up the origin, with $D = \pi^{-1}(0)$ being a projective line representing the set of all (complex) directions through the origin. If $\omega$ defines a holomorphic foliation over $\mathbb{C}^2$, then $\omega_1 := \dfrac{1}{x^{\nu}} E^{\ast} \omega$ represents the blow-up foliation, where $x=0$ is a local equation of $D$ and $\nu\in \NN$ is the greatest power of $x$ appearing in the expression of $E^{\ast} \omega$. In fact, $\nu$ coincides with the order $\nu(\omega)$ of $\omega$ in the non-dicritical case (when $D$ is an invariant curve of the foliation) and with $\nu(\omega) + 1$ in the dicritical case (non-invariant $D$). After a finite number of point blow-ups, a germ of foliation is transformed in another foliation, defined over a different ambient space, with singularities of simple (or reduced) type, i.e. generated locally by a germ of vector field $\XX$ with non zero linear part, such that the eigenvalues $\lambda_1$, $\lambda_2$ of this linear part verify $\lambda_1\neq 0$ and $\lambda_1/\lambda_2\notin \QQ_{>0}$. A partial proof of this fact may be read in \cite{Seidenberg,vdE}, and a complete one in \cite{MM,CS2}. Two main cases are under consideration:
\begin{enumerate}
\item $\lambda_2\neq 0$. Hyperbolic case.
\item $\lambda_2=0$. Saddle-node case.
\end{enumerate}

In the first case, if $\lambda_2/\lambda_1 \notin \RR$, the foliation is said to be in the Poincaré domain \cite{poincare}, and it is analytically linearizable. The same happens if $\lambda_2/\lambda_1 \in \RR \setminus \QQ$ is not well approximated by rationals, and no small divisors appear. In the presence of resonances ($\lambda_2/\lambda_1 \in \QQ_{<0}$), the situation becomes more complicated: it appears that two formally equivalent foliations are not necessarily analytically equivalent. The analytic classification involves the construction of a large moduli space \cite{MR2}. It is interesting to remark here that the formal normal forms of foliations are not only analytic but also polynomial. However, this does not mean that the normalizing transformations are analytic, nor that formal and analytic classifications coincide.

In the saddle-node case, similar things occur. Such a foliation is formally determined by two numbers (the Milnor number of the singularity and a residue, which is Camacho-Sad index of the weak separatrix, see below), and the analytical classification is again much more complicated, involving a certain infinite-dimensional space. Again, formal normal form is polynomial.

There are several techniques that have proven useful in the study of the analytic classification and the analytic transformations of these objects. One of them is widely used by E. Str\'{o}\.{z}yna and H. \.{Z}o{\l}{\c{a}}dek in \cite{SZ,SZ2}: it consists in a clever use of transformations involving vector fields, complementary subspaces and homological equations in order to guarantee convergence. Another one, of a more geometrical nature, is the holonomy, introduced by J.-F. Mattei and R. Moussu \cite{MM}. Let us recall here that any reduced singularity has exactly two (formal or analytic) smooth invariant  varieties through the singularity, transversal to each other, whose tangent spaces follow the eigendirections of the linear part. These varieties, called separatrices are analytic when they are tangent to the eigendirection determined by a nonzero eigenvalue, so in the saddle-node case a formal, divergent separatrix may appear. Associated to a leaf of the foliation, a holonomy group can be constructed: take a non-singular point $P$ on the leaf and a small germ of transversal $\Sigma$ to the foliation through $P$. Any loop $\gamma$ on the leaf based at $P$ can be lifted thanks to the existence of a transversal fibration locally around $\gamma$, to a new path $\tilde{\gamma}$ with initial point on $\Sigma$ (call it $Q$) and end point $h(Q)$. Identifying $\Sigma\cong (\CC,0)$ a local diffeomorphism $h:(\CC,0)\rightarrow (\CC,0)$ is defined. The set of these diffeomorphisms is a group called the holonomy group of the leaf.

This construction can be made locally around a singular point, taking a separatrix as the leaf. In the resonant case, it is well known that if two formally equivalent singularities have analytically equivalent holonomy maps computed over one of these separatrices, the germs of foliations themselves are analytically equivalent (\cite{MM,MR2,Loray}). It is also the case with the separatrix associated to the non-zero eigenvalue (strong separatrix) of a saddle-node. There is an important difference concerning the proofs, anyhow: while in the resonant case the conjugation between both foliations can be constructed extending the holonomy, due to the presence of a transversal fibration to the foliation around the singular point, this technique cannot be carried out in the saddle-node case, due to the lack of an appropriate transversal fibration.

Concerning the analytic classification of non-reduced germs of foliations, in this paper we will  focus in the nilpotent case. These are foliations  generated locally  by a vector field with a non-zero nilpotent linear part. It is well-known that such a foliation, in appropriate analytic coordinates, can be reduced to the Takens normal form \cite{Takens,SZ,Loray1}
$$
d(y^2 + x^n) + x^{p} U(x) dy,
$$
where $U(0)\neq 0$, $n$, $p\in \NN$, $n\geq 3$, $p\geq 2$. Several different cases need to be considered:
\begin{enumerate}
\item $2p>n$. Generalized cusp.
\item $2p=n$. Generalized saddle.
\item $2p<n$. Generalized saddle-node.
\end{enumerate}

We borrow these names from the works of E. Str\'{o}\.{z}yna and H. \.{Z}o{\l}{\c{a}}dek \cite{Strozyna,SZ,SZ2}.
The generalized cusp case was first studied by R. Moussu \cite{Moussu} when $n=3$, and by D. Cerveau and R. Moussu \cite{CM} in the general case. R. Moussu introduces the notion of \textit{vanishing} holonomy (called projective holonomy in subsequent works), which is the holonomy of a certain component of the exceptional divisor that appears after reduction of singularities as the main analytic invariant.

Later on, R. Meziani \cite{Meziani} considers the case $2p=n$. He studies two subcases. When $U(0)\neq \pm 4$ and $U(0)$ verifies certain arithmetical condition, the situation is quite close to Case 1. If $U(0)= \pm 4$, after reduction of singularities there is only one singular point which is not a corner, of saddle-node type with strong separatrix transversal to the exceptional divisor. The holonomy of this separatrix is an invariant for analytic classification. From this fact, rigidity phenomena (formal-analytic) may be studied.

When $2p<n$, we are in the generalized saddle-node case, studied in \cite{BMS} and in \cite{Strozyna}. It remains to be studied the case when $2p=n$, $U(0)\neq \pm 4$ and the arithmetical condition mentioned above fails. Under these conditions, after $p$ blow-ups, two singularities appear in the last component of the exceptional divisor (besides the corners): a resonant one and a second one that may fall in one of the following two categories:
\begin{enumerate}
\item Either it is dicritical.
\item Or it is of Poincar\'{e}-Dulac type.
\end{enumerate}

A non-dicritical singularity is called of Poincaré-Dulac type when defined locally by a vector field $\XX$ having a linear part with two eigenvalues  $\lambda_1$, $\lambda_2$ such that $\lambda_2/\lambda_1=m\in \NN_{\geq 2}$. This is a standard name, used widely in the literature regarding holomorphic foliations of codimension one, as in \cite[p. 52]{Scardua}, \cite[p. 49]{CS2} or \cite{LNS}. These are singularities belonging to the Poincaré domain, which were studied by H. Dulac \cite{Dulac}, and its name honors these two outstanding mathematicians. Let us remark that the term \textit{Poincaré-Dulac} is also used by other authors in order to describe a reduction procedure in some cases, see for instance \cite{IY,Zoladek}.

Case 1 (dicritical) was studied by R. Meziani and P. Sad \cite{MS} from the point of view of characterizing when these foliations have a meromorphic first integral (results later generalized in dimension three by P. Fernández, J. Mozo and H. Neciosup \cite{FMN3}). In this paper we are focusing in the Case 2 above, which will be called, in coherence with the names given by the aforementioned authors,  \textit{generalized Poincar\'{e}-Dulac}. The objective will be twofold:\begin{enumerate}
\item To study the analytic classification for these generalized Poincar\'{e}-Dulac singularities of holomorphic foliations,  from a geometrical point of view, using projective holonomy.

\item To study the projective holonomy groups, formal-analytic rigidity and realizability of these groups.
\end{enumerate}

Note that after  $p$ blow-ups, the foliation is not desingularized. In fact, it appears a singularity of Poincar\'{e}-Dulac type, with a normal form $xdy-(my+x^m)dx$, for some $m\in \NN$, $m\geq 2$. This integer $m$ is the only analytic invariant. It is necessary to blow-up $m$ additional points for this singularity to be reduced; at the end, a saddle-node appears in a corner, and there are no more separatrices. Nevertheless, in our research we will not need to perform these last $m$ blow-ups, and the analytic classification will be obtained from the holonomy of the $p$th component of the exceptional divisor. This fact (no need to desingularize completely the foliation) is new for this class of nilpotent singularities.

Generalized saddles have also been studied by E. Str\'{o}\.{z}yna and H. \.{Z}o{\l}{\c{a}}dek in \cite{SZ2} from a completely different point of view. We will relate our results to theirs when studying the rigidity of the projective holonomy group.

Another point to be highlighted is that generalized Poincar\'{e}-Dulac singularities are not second type foliations in the sense of \cite{Mattei-Salem}. This situation also occurs in Meziani's case $U(0)=\pm 4$, where a saddle-node in ``bad'' position (weak separatrix contained in the divisor) appears. Nevertheless, we will not explore this fact.

The structure of the paper is as follows. In Section \ref{singularidades-nilpotentes}, a review of nilpotent singularities in dimension two is given, and generalized Poincar\'{e}-Dulac singularities are introduced. The main part of Section \ref{clasificacion-analitica} is devoted to the proof of the main result of the paper, which is the following one:

\begin{teorema1}
Let $\FF_1$, $\FF_2$ be two germs of generalized Poincar\'{e}-Dulac holomorphic  foliations, formally equivalent, with $n=2p$. Assume that $H_i$ is the holonomy group of the $p$th component of the exceptional divisor obtained during the process of reduction of singularities for $\FF_i$ ($i=1,2$). If $H_1$, $H_2$ are analytically conjugated, the foliations are also analytically conjugated.
\end{teorema1}

Finally, in Section \ref{rigidez} we study the realizability of the groups appearing as projective holonomy, and also, the formal-analytic rigidity of generalized Poincar\'{e}-Dulac foliations. We will make convenient connections with related results stated in \cite{SZ2}.

\section{Nilpotent singularities in dimension two. Generalized Poincar\'{e}-Dulac singularities.} \label{singularidades-nilpotentes}

Let us consider a germ of a singularity of a nilpotent holomorphic foliation $\FF$ of codimension one, over an ambient space of dimension two, defined by a 1-form $\omega$ on $(\CC^2,\bo)$. According to F. Takens \cite{Takens}, a formal \textit{prenormal} form for $\omega$ is
\begin{equation} \label{formaprenormal}
\eta =d (y^2+x^n)+ \alpha x^p U(x)dy,
\end{equation}
where $p$, $n\in \NN$, $n\geq 3$, $p\geq 2,$ $\alpha\in \CC^{\ast}$, $U(x)\in \CC [[x]]$, $U(0)=1$.  E. Str\'{o}\.{z}yna and H. \.{Z}o{\l}{\c{a}}dek \cite{SZ}, using direct computations, and independently F. Loray \cite{Loray} with geometrical arguments  have shown that $\eta$ can be obtained from $\omega$ after an analytic change of coordinates, so we may assume that $U(x)\in \CC \{ x\}$. This expression of the 1-form generating the foliation will be our starting point.

As we have stated in the Introduction, we are interested in the generalized saddle case, so here, and in the rest of the paper, we shall assume that $n=2p\geq 4$. Let us briefly describe, without many details, the reduction of singularities of $\FF$. After a chain of $p$ point  blow-ups, $p$ components $\lista{D}{p}$ of the exceptional divisor appear in this order. The singular points are the corners (intersection of two components of the exceptional divisor), and one or two singular points situated in $D_p$. If we take local coordinates such that this chain of blow-ups is obtaining replacing $y$ by $x^pz$ and eliminating powers of $x$ from the expression, the singularities correspond with the roots of the polynomial $2(z^2+1) +\alpha z=0$. Indeed, the transformed foliation $\FF_p$ is generated locally in these coordinates by
$$
\tilde{\omega}= p (2(z^2+1) + \alpha z U(x)) dx + x(2z + \alpha U(x)) dz,
$$
$x=0$ being the equation of the component of the divisor $D_p$. Several possibilities arise:
\begin{enumerate}
\item $\alpha=\pm 4$. In this case, only a singular point appears, which turns out to be a saddle-node, with strong variety transversal to $D_p$ and weak variety contained in the divisor. The foliation is reduced, and this case has already been studied by R. Meziani \cite{Meziani-tesis,Meziani}.
\item $\alpha \neq \pm 4$. Two singular points appear, corresponding to the different roots $z_1$, $z_2$ of the polynomial $2(z^2+1)+\alpha z =0$. If $\dfrac{\sqrt{\alpha^2-16}}{\alpha}\notin \QQ \cap (-1,1)$, both singularities are simple. This case has also been studied by R. Meziani.
\item If $\dfrac{\sqrt{\alpha^2-16}}{\alpha}\in \QQ \cap (-1,1)$, one singularity (say, $z_2$) is simple, resonant, and the quotient of the eigenvalues of the other singularity ($z_1$) is positive rational. Under these conditions, this singular point may be either dicritical (this case was studied by R. Meziani and P. Sad in \cite{MS}) or a singularity of Poincaré-Dulac type.
\end{enumerate}

In order to simplify computations, we begin centering the coordinates at  $z_1$, after a translation $z\mapsto z+z_1$. The foliation is locally generated by
\begin{equation} \label{PDafterbu}
p(2z (z+z_1-z_2) + \alpha (z+z_1) \tilde{U}(x)) dx+ x(2z-2z_2+\alpha \tilde{U}(x))dz, 
\end{equation}
where we have written $U(x)=1+\tilde{U}(x)$, $\tilde{U}(0)= 0$. The matrix of the linear part of the dual vector field is
$$
\begin{pmatrix}
2z_2 & 0 \\
\ast & 2p (z_1-z_2)
\end{pmatrix}.
$$
We need to impose that this singularity is of Poincar\'{e}-Dulac type, its only separatrix being the exceptional divisor. If this occurs, there exists $m\in \NN$, $m\geq 2$, such that
$$
p\cdot \frac{z_1-z_2}{z_2}=m.
$$
This equality imposes strong conditions on $\alpha$. Namely, we must have
\begin{equation} \label{condicion_gPD}
\alpha= -\frac{2 (m+2p)}{\sqrt{p (m+p)}}.
\end{equation}
Given such a value for $\alpha$, it is not straightforward to distinguish whether  we are facing a dicritical singularity or a  Poincar\'{e}-Dulac one. By Poincaré-Dulac Theorem such a singularity must be  analytically equivalent to 
\begin{equation} \label{PDm}
 xdz - (mz+ax^m)dx, \tag{PDm},
\end{equation}
for some $a\neq 0$, as we are in the Poincaré domain (and if $a=0$ the singularity would be dicritical). 

 The reduction of singularities is completed after blowing-up points $m$ additional times: if we are in the generalized Poincar\'{e}-Dulac case, only in the corners new singularities may appear, all of them resonant except for the last one, which turns out to be a saddle-node, with both separatrices (weak and strong) contained in the divisor.

\begin{ejemplo}
Assume that $U(x)= 1+ \tilde{U}(x)$, with $\nu (\tilde{U} (x))\geq m$, i.e., $U(x)= 1+ax^{m} + h.o.t.$ After $m-1$ blow-ups we have the foliation of equation
\begin{multline*}
\left[
\left( 2z (x^{m-1}z +z_1-z_2) + \alpha z \tilde{U}(x) + \alpha z_1 \frac{\tilde{U} (x)}{x^{m-1}} \right) \right. \\ \left. + 
(m-1) z (2x^{m-1} z +2z_1 + \alpha + \alpha \tilde{U} (x)) \right] dx \\
 + x [2x^{m-1} z + 2z_1 + \alpha + \alpha \tilde{U} (x)] dz,
\end{multline*}
where the matrix of the linear part is
$$
\begin{pmatrix}
2z_2 & 0 \\ p\alpha z_1 a & 2z_2
\end{pmatrix}.
$$
If $a=0$, we are in the dicritical case, and if $a\neq 0$ in the generalized Poincar\'{e}-Dulac case. Both cases, then, may appear naturally.
\end{ejemplo}

\begin{ejemplo}
Take $p=2$, $\alpha=-5$, $U(x)=1+bx$. In this case, $z_1$, $z_2$ are the solutions of $2z^2-5z+2=0$, which are $2$ and $\frac{1}{2}$. Here, $m= 6$. After the first two blow-ups, centering in $z_1=2$, we have the 1-form
$$
\omega_1= (4z^{2} + 6z -10 bx (z+2))dx + x(2z-1-5bx) dz.
$$
It is necessary  to blow-up 5 more times to recognise if we have a dicritical singularity or a generalized Poincar\'{e}-Dulac one. After these 5 blow-ups, centering the coordinates in the only singular point appearing in the smooth part of the divisor, we get a singularity with linear part\footnote{Computations done with Maple 2021.2, licensed for Valladolid University.}
$$
(5934060 b^6 x+z) dx-xdz,
$$
so, it is a generalized Poincar\'{e}-Dulac singularity if $b\neq 0$.

\end{ejemplo}

%
%

\section{Holonomy and analytic classification} \label{clasificacion-analitica}

The main objective of this section is to show that the holonomy of a convenient component of the exceptional divisor analytically characterizes the foliation when considering generalized Poincaré-Dulac foliations. In previous works addressing this problem for other classes of nilpotent singularities in two dimensions, the key tool was the projective holonomy of a well chosen component of the exceptional divisor obtained after the full reduction of singularities. On the contrary, we will use the projective holonomy of the component of the exceptional divisor obtained after the first $p$ blow-ups. Thus, we will skip the remaining $m$ components and work directly with the Poincaré-Dulac singularity, without completing the reduction process.

The main result of this section, and one of the main objectives of the paper, is the following one:
\begin{teorema}
\label{teoremaprincipal}
Let $\FF_1$, $\FF_2$ be two germs of generalized Poincar\'{e}-Dulac holomorphic  foliations as in \eqref{formaprenormal}, formally equivalent, with $n=2p$. Assume that $H_i$ is the holonomy group of the $p$th component of the exceptional divisor obtained during the process of reduction of singularities for $\FF_i$ ($i=1,2$). If $H_1$, $H_2$ are analytically conjugated, the foliations are also analytically conjugated.
\end{teorema}

Let us postpone the proof to the end of this section. Remark that
a closer study of the holonomy group will allow to go beyond this result and to show in Section \ref{rigidez} that, in fact, there is a formal-analytic rigidity phenomenon for this type of foliations. Before that, let us precise the structure of this group, which will we useful during the proof of Theorem \ref{teoremaprincipal}.

\subsection{Structure of the holonomy group} \label{estructura}

Call $H$ the projective holonomy group of the $p$th component of the exceptional divisor generated in the process of reduction of singularities from a generalized Poincar\'{e}-Dulac singularity. Computed over a transversal $\Sigma$ to $D_p$, it is generated by two elements $h_1$, $h_2$, where $h_1$ is the local holonomy of a Poincar\'{e}-Dulac singularity, and $h_2$ the local holonomy of a resonant one.

From the local model $\eqref{PDm}$ of a Poincar\'{e}-Dulac singularity, it is straightforward to compute the holonomy, which turns out to be
$$
e^{\frac{2\pi i}{m}} \exp \left( -\frac{2\pi i}{m} \cdot \frac{x^{m+1}}{m+x^{m}} \cdot \pa{x} \right).
$$
As the singularity $P_1$ is analytically equivalent to $\eqref{PDm}$, we have that $h_1 (x)= e^{\frac{2\pi i}{m}} \exp (\YY) $, where $\YY$ is an analytic 1-dimensional vector field. The composition $h_0:= h_1\circ h_2$ is the holonomy of the corner singularity $D_{p-1}\cap D_p$, which is periodic (\cite{MM}), $h_0^{[p]}=id$. So, there is an analytic coordinate $z$ such that $h_0(z)= \lambda z$, $\lambda^{p}=1$. In this coordinate, $h_1 (z)= \mu \exp (\YY)$, $\mu^{m}=1$, and $h_2 (z)= h_0 \circ h_1^{-1}= \dfrac{\lambda}{\mu} \exp (-\YY)$.

So, there exist analytic coordinates such that $H= \langle \mu \exp (\YY), \dfrac{\lambda}{\mu} \exp (-\YY)\rangle$, with $\YY$ a holomorphic vector field of order $m+1$, $\mu^{m}=1$, $\lambda^{p}=1$. We will use this structure along the proof of Theorem \ref{teoremaprincipal}.

\subsection{Previous results and proof of the main Theorem}

Before showing the theorem, we will need some previous results and remarks. For a generalized Poincar\'{e}-Dulac singularity, recall from \eqref{PDafterbu} that after $p$ blow-ups we find that it is generated around the Poincar\'{e}-Dulac singularity by
\begin{equation}
\label{PD}
p(2z (z+z_1-z_2) + \alpha (z+z_1)\tilde{U}) dx+ x(2z-2z_2+\alpha \tilde{U} (x))dz,
\end{equation}
where $m=p\cdot \dfrac{z_1-z_2}{z_2}\in \NN$, $m\geq 2$. Dividing \eqref{PD} by the unit $2z-2z_2+\alpha \tilde{U}(x)$,  the following local generator can be taken:
$$
\omega_{PD}= xdz -\left[ z\cdot \frac{m+z_1 z}{1-z_1 z} +x A(x,z) \right] dx.
$$
We will assume that it is analytically equivalent to $xdz-(mz+x^m)dx$. Let us observe that the fibration $x=const$ is transversal to the foliation away from the divisor $x=0$. Indeed, $dx\wedge\omega_{PD}= xdx\wedge dz $. 

A vector field, dual of $\omega_{PD}$, generating the foliation, is
$$
\XX = x \pa{x} + (mz + a(x,z))\pa{z},
$$
for certain holomorphic function $a(x,z)$ not detailed here. We want to investigate the effect of a change of variable $u= z+\varphi_{k}(x,z)$, where $\varphi_{k}$ is a homogeneous polynomial of degree $k$, which transforms $\XX$ in a similar vector field $\YY= x\pa{x} + (mu + b(x,u))\pa{u}$, for some $b(x,u)$. Both vector fields $\XX$ and $\YY$ are related by
\begin{equation}
\label{relacion}
x\der{\varphi_k}{x} + mz \der{\varphi_k}{z} -m \varphi_k =b(x,z+\varphi_k)- a (x,z)-a(x,z)\varphi_k.
\end{equation}
The homogeneous terms of degree smaller than $k$ for $a$, $b$ agree. The degree $k$
 terms, $a_k$, $b_k$, satisfy the relation
\begin{equation}
 \label{krelacion}
 x\varphi_k + mz \varphi_{k,z} - m\varphi_k = b_k (x,z)- a_k (x,z),
 \end{equation} 
 where $\varphi_{k,z}$ denotes the partial derivative $\der{\varphi_k}{z}$, as usual. Writing $\dis \varphi_k (x,z)= \sum_{i+j=k} \varphi_{k,i,j} x^i z^{j}$, the left hand side of \eqref{krelacion} is 
 $$
 \sum_{i+j=k} (i+m (j-1)) \varphi_{k,ij}x^{i}z ^{j}.
 $$
Therefore, if $k\neq m$, $\varphi_k (x,z)$ can be chosen in such a way that $b_k (x,z)=0$. As a consequence, after a convergent (in the $\mathfrak{m}$-adic topology) sequence of polynomial transformations of this type, $b(x,z)$ can be reduced to $\varepsilon x^{m}$, $\varepsilon\in \CC$. If we only make a finite number of such transformations, we can assume that
 $$
 \XX= x\pa{x} + (mz+\varepsilon x^{m}+ g(x,z)) \pa{z}, $$
 with $\nu (g(x,z))= r\geq m+1$. 
 
 \begin{propo} \label{fibrado}
Given a vector field $
 \XX= x\pa{x} + (mz+\varepsilon x^{m}+ g(x,z)) \pa{z} $, $\varepsilon\neq 0$, $\nu (g)=r\geq m+1$, there exists a holomorphic conjugation $z=u + \varphi (x,u)$ transforming $\XX$ in its normal form
 $$
 \YY_m=x\pa{x} + (mu+\varepsilon x^m) \pa{u}.
 $$
 In other words, the transformation which conjugates $\XX$ with its normal form can be chosen in such a way that  respects the fibration $dx=0$.
 \end{propo}
 
 \begin{proof}
 The proof is similar to other proofs of related results, which can be read, for instance, in \cite{IY}. We shall include it here for the sake of completeness, adapted to our case.
 
 According to \eqref{relacion}, we have that
 $$
 x\varphi_{x} + mu \varphi_u -m \varphi = g(x,u+\varphi ) - \varepsilon x^m \varphi_u.
 $$
Denote $ad (\varphi)=x\varphi_x+ mu \varphi_u - \varphi$, $T_1 (\varphi) = g(x,u+\varphi)$, $T_2 (\varphi)= -\varepsilon x^{m} \varphi_u$. We have that 
\begin{equation} \label{puntofijo} 
ad (\varphi)= T_1 (\varphi) + T_2 (\varphi).\end{equation}
The objective is to find a convergent solution of \eqref{puntofijo}. If $\rho>0$, $\BB_{\rho}$ will denote the space of holomorphic and bounded functions $\varphi (x,u)= \sum_{i,j} \varphi_{i,j} x^{i} u^{j}$ on $D(0;\rho) \times D(0;\rho)$ such that $\dis \sum_{i,j} \norm{\varphi_{ij}} \rho^{i+j} < +\infty$, provided with the norm $\norma{\varphi}_{\rho} := \sum_{i,j} \norm{\varphi_{ij}} \rho^{i+j}$, which makes $\BB_{\rho}$ a Banach space. Previous operators act on the subspace  $\BB_{m+1,\rho}\subseteq \BB_{\rho}$ of series whose order is greater or equal than $m+1$. Let us observe that
$$
\norma{ad^{-1}\circ T_2 (\varphi)}_{\rho} =\sum_{i,j} \frac{\norm{\varepsilon} \norm{\varphi_{ij}} j\rho^{m-1}}{i+m(j-1)} \rho^{i+j} \leq \frac{\norm{\varepsilon} \rho^{m-1}}{m^2} \cdot \norma{\varphi}_{\rho},
$$ 
if the order of $\varphi$ is at least $m+1$, given that
$$
\max \left\{ \frac{j}{i+m (j-1)} \mid\, i,j\geq 0,\ i+j\geq m+1\right\} = \frac{m+1}{m^2}.
$$
On the other hand,
$$
T_1 (\varphi_2)- T_1 (\varphi_1) = (\varphi_2 -\varphi_1) \int_0^1 \der{g}{u} (x,u+t \varphi_2 + (1-t) \varphi_1) dt.
$$
As $\dis \nu \left( \der{\varphi}{u} \right)\geq m$, 
$$
\norma{\der{g}{u} (x,u)}_{\rho'} \leq K \cdot (\rho')^{m}, \text{ for certain } K>0,\ 0< \rho'\leq \rho.
$$
We have that
$$
\norma{u+t \varphi_2 + (1-t) \varphi_1 }_{\rho} \leq \norma{u}_{\rho} + \max ( \norma{\varphi_1}_{\rho}, \norma{\varphi_2}_{\rho}) \leq \rho + c\rho^{m+1} 
$$
for some $C>0$. As $ad^{-1}$ is a bounded operator (because $i+m(j-1)\geq 1$ when $i+j\geq m+1$), previous computations show that, for $\rho$ small enough, equation \eqref{puntofijo} has a solution (i.e., $\varphi = ad^{-1} (T_1 + T_2) (\varphi)$ has a fixed point), which solves the problem.

 \end{proof}

\begin{proof}[Proof of Theorem \ref{teoremaprincipal}]
In this proof , we will take some ideas of previous related works, as \cite{CM,Meziani,Strozyna}. Let $H_i$ be the projective holonomy group corresponding to the foliation $\FF_i$ ($i=1,2$), $H_i= \langle h_{i1}, h_{i2}\rangle$, where $h_{i1}$ is the local holonomy of the Poincar\'{e}-Dulac singularity ($P_{i1}$), and $h_{i2}$ of the hyperbolic saddle ($P_{i2}$). As $h_{i1}$, $h_{i2}$ are conjugated, both Poincar\'{e}-Dulac singularities $P_{11}$, $P_{21}$ are of the same formal type (same $m$), so they are of course analytically conjugated by a local biholomorphism $\Phi (x,z)= (x,\varphi (x,z))$, defined over a convenient neighbourhood $V_{11}$ of $P_{11}$. Remark that, according to Proposition \ref{fibrado}, $\Phi$ can be chosen fibered in $x$.

Let $\Sigma $ be a germ of transversal to $D_p$ for $\FF_1$, $\Sigma \subseteq V_{11}$, parametrized by $(x,s_1)$, $x\in (\CC, 0)$ and $\Sigma' = \Phi (\Sigma )= (x, \varphi (x,s_1))$. Let $h= \Phi \mid_{\Sigma} : \Sigma\rightarrow \Sigma'$ be the restriction of $\Phi$ to this transversal: in these coordinates, in fact, $h(x)=x$. Let $g: \Sigma \rightarrow \Sigma'$ be a conjugation between $H_1$ and $H_2$. By construction, $h^{-1}\circ g : \Sigma \rightarrow \Sigma$ commutes with $h_{11}$.

In these coordinates, $h_{11}= \mu \exp (\YY)$, with $\mu^m=1$, and $h_{12}= \frac{\lambda}{\mu} \exp (-\YY)$, $\lambda^{p}=1$. As $h(x)=x$ conjugates $h_{11} $ with $h_{21}$, then $h_{21}= \mu \exp (\YY)$ and consequently, $h_{22}= \frac{\lambda}{\mu} \exp (-\YY)$. As a conclusion, the restriction $h$ of $\Phi$ to $\Sigma$ defines a conjugation between $H_1$ and $H_2$.

The fibration $dx=0$ is transversal to $\FF_1$ out of the exceptional divisor and to the curve $2z+\alpha U(x)=0$, which cuts the divisor in a point of coordinate $s_0= -\frac{\alpha}{2}$. This is a regular point for the foliation.

Let $P= (x,z)$ be a point, neither on the separatrix through the saddle $P_{22}$, nor on the curve $2z+ \alpha U(x)=0$. Take a path $(x(t), z(t))$ on the leaf $L_p$ through $P$, joining $P$ with $P'\in V_{11}$. As $dx=0$ is transversal to the foliations $\FF_1$, $\FF_2$, there exists $z_{2}(t)$ such that:
\begin{enumerate}
\item $z_2 (1) = \varphi (x(1), z(1))= \varphi (P')= Q'$.
\item $(x(t), z(t))$ is a path on the leaf $L_{Q'}$ of $P_2$ passing through $Q'$.
\end{enumerate}
Denote $Q= (x(0), z_2 (0))$. The map assigning $Q$ to $P$ defines a conjugation between the foliations in an open set, that we want to extend to neighbourhoods of the singularity $P_{12}$, of the point $z_0$ and of the corner $D_p \cap D_{p-1}$.

We study first the saddle $P_{12}$. The fibration $dx=0$ is transversal to the foliation and to the separatrix in a neighbourhood of $P_{12}$. So, the procedure described above turns out to be the lifting of the projection of the path $(x(t), y(t))$ over the separatrix through $P_{12}$. Using th same arguments as in  \cite{MM,MR2,Loray}, it is bounded along this separatrix, so it can be extended. This extension of $\Phi$, in a neighbourhood of $P_{12}$ has locally an analytic expression
$$
\Phi (x,z)= \left( x, \sum_{n=0}^{\infty} \varphi_{n} (x) z^{n} \right),
$$
(after a convenient translation) with $\varphi_n (x)\in \OO (D\setminus \{ 0\})$, $D$ being a fixed disk around $0$. We want to extend this map also to the exceptional divisor $x=0$. Consider the local foliations around $P_{12}$, $P_{22}$, defined by local 1-forms $\tilde{\omega_i} $ ($i=1,2$) of the form (following \cite{MR2})
$$
\tilde{\omega}_i = xdz + (\xi z + B_i (x,z))dx,
$$
where $\xi= \dfrac{m+p}{mp}$.
$\tilde{\omega}_1$, $\tilde{\omega}_2$ are formally equivalent, so they share a formal normal form as
$$
\omega_{\lambda,s} =xy v^{k} \left( \frac{dx}{x} + \left( \lambda + \frac{1}{v^{s}} \right) \frac{dv}{v} \right),
$$
where $v= x^{m+p} z^{mp}$. There are formal transformations $\Psi_i$, $i=1,2$, such that $\Psi_i^{\ast} \omega_{\lambda,s} \wedge \tilde{\omega}_i=0$, and these transformations are, in fact, transversely formal in $z$, i.e., 
$$
\Psi_i (x,z)= \left(  x, \sum_{k=1}^{\infty} \Psi_{ik} z^{k} \right),
$$
with $\psi_{ik}$ holomorphic in a common disk of convergence. Consider now the (transversely formal) map $\Psi_2\circ \Phi \circ \Psi_1^{-1}=: \Psi$, verifying $\Psi^{\ast} \omega_{s,\lambda} \wedge \omega_{s,\lambda}=0$. By construction,
$$
\Psi (x,y)= \left( x,\sum_{k=1}^{\infty} \Psi_{k} (x) y^{k} \right),$$
where $\Psi_k \in \OO (D\setminus \{ 0\})$. The isotropy group of $\omega_{s,\lambda}$ can be explicitely described, and it can be seen that $\Psi_{k} (x)$ can, in fact, be analytically continued to 0. As a consequence, $\varphi_k (x)$ also can be continued and $\Phi (x,z)$ extends to $\{0 \} \times (\CC,0)$, i.e., to the exceptional divisor.

Consider now a neighbourhood of $s_0$ (assume, for the sake of simplicity, that it is the origin). As it is a regular point, a holomorphic first integral exists, having $x=0$ as a level and being transversal to $dx=0$ beyond the exceptional divisor. Denote $F_i (x,z)= x U_i (x,z) $, $U_i (0,0)=1$ this first integral, $i=1,2$, for each of the foliations $\mathcal{F}_i$, defined on $\norm{x}<\varepsilon$, $\norm{z} <\eta$, $U_i (0,z)\neq 0$, $\der{U_i}{z} (0,0) \neq 0$. The map $\Psi_i (x,z) = (x, U_i (x,z)-1)$ defines a biholomorphism in a neighbourhood of $z_0$ transforming the first integral in $x(1+z)$.

Consider now the map $\Psi_2 \circ \Psi \circ \Psi_1^{-1}= \theta (x,z)$, $\norm{x}<\varepsilon$, $0<\varepsilon_1<\norm{z}<\varepsilon_2$, defined on an annulus close to $z_1$ and preserving $dx=0$: $\theta (x,z) = (x,\theta_2 (x,z))$. The first integral $x(1+z)$ is transformed on $x(1+z) V(x(1+z))$, where $V(0)\neq 0$:
$$
x(1+\theta_2 (x,z))=x (1+z) V(x(1+z)),
$$
so $\theta_2 (x,z)= (1+z) V(x(1+z))-1$, well defined on a neighbourhood of $z_0$: the conjugation $\Psi (x,z)$ extends to a neighbourhood of this point. This last argument is similar to the one detailed by Meziani \cite{Meziani-tesis,Meziani}.

Finally, consider the corner $D_{p-1}\cap D_{p}$, linearizable, having a holomorphic first integral in a neighbourhood of that point. Standard arguments as in \cite{CM,Meziani-tesis,Meziani} allows to extend the conjugation to a neighbourhood of this point, so to a full neighbourhood of $D_p$, and similarly to a neighbourhood of the exceptional divisor. Collapsing this exceptional divisor in a point and using Hartogs' Theorem, we end the construction of the conjugation between both foliations.
\end{proof}

 \section{Realizability and Rigidity} \label{rigidez}
  
We have seen in Section \ref{estructura} that the  projective holonomy group of a generalized Poincar\'{e}-Dulac singularity has the form $H=\langle \mu \exp (\YY), \dfrac{\lambda}{\mu} \exp (-\YY) \rangle$, where $\YY$ is an analytic vector field, $\mu^m=1$, $\lambda^{p}=1$. In this section we will address two problems: 
\begin{enumerate}
\item Are all these groups realizable as the projective holonomy group of a generalized Poincar\'{e}-Dulac singularity?
\item Are these group rigid, in the sense of formal-analytic rigidity?
\end{enumerate}

Regarding the first problem,
consider a vector field $\YY$ analytically equivalent to $-\dfrac{2\pi i}{m}\cdot \dfrac{x^{m+1}}{m+x^{m}}\cdot \dfrac{\partial}{\partial x}$, where $\mu$ is a primitive $m$th root of unity and $\lambda$ is a primitive $p$th root. Define $h_1(x) = \mu \exp(\YY)$, $h_2(x) = \frac{\lambda}{\mu} \exp(-\YY)$, and assume that $h_0 := h_1 \circ h_2$ is linearizable. A local model of a foliation having $h_1(x)$ as its holonomy can be easily constructed: if $\varphi_{\ast} \YY = -\dfrac{2\pi i}{m}\cdot \dfrac{x^{m+1}}{m+x^{m}}\cdot \dfrac{\partial}{\partial x}$, then take $\Psi^{\ast} (xdy - (my + x^m)dx)$, where $\Psi(x,y) = (\varphi(x), y)$. According to \cite{MR2}, a local model of a resonant singularity having $h_2$ as holonomy can also be found. Therefore, the construction in \cite{Lins} can be applied to find an analytic surface $M$, an embedded projective line $D \subseteq M$, and a foliation $\FF$ on $M$ such that $D$ is a leaf of $\FF$ with three singular points and prescribed holonomy. By the Camacho-Sad Index Theorem \cite{Camacho-Sad}, the self-intersection of $D$ must be $-1$, so these projective holonomy groups are always realizable.

Next, we will study rigidity.
First, we need to study the groups that appear as projective holonomy groups in a generalized Poincaré-Dulac singularity. From previous results, the projective holonomy group of a generalized Poincaré-Dulac singularity has the form $H = \langle \mu \exp(\YY), \dfrac{\lambda}{\mu} \exp(-\YY) \rangle$, where $\YY$ is an analytic vector field, $\mu^m = 1$, and $\lambda^p = 1$.
 
 The vector field $\YY$ is analytically equivalent to $\dfrac{x^{m+1}}{m+x^{m}} \dfrac{\partial}{\partial x}$. We know that the linear map  $\alpha z$ commutes with $\exp (\YY)$ if and only if $\alpha^{m}=1$. Taking this into account, if $H$ is abelian, it turns out that
$$
\exp (\YY) \circ \lambda \exp (-\YY)= \lambda z,
$$
so, $\exp (\XX)$ commutes with $\lambda z$, which means that $\lambda^{m}=1$. This can happen only if $p$ divides $m$.

Assume that $H$ is not abelian. If $H$ would be solvable, it should be formally equivalent to a subgroup of
$$
\left\{a\cdot \exp \left( tx^{m+1}\cdot \pa{x} \right) |\, t\in \CC,\ a\in \CC^{\ast}\right\},
$$
so $\YY$ should be equivalent to $x^{m+1} \pa{x}$, which is not true. We arrive to the following dichotomy:

\begin{quote}
The projective holonomy group of a generalized Poincar\'{e}-Dulac singularity is:
\begin{itemize}
\item Either abelian, if $p$ divides $m$.
\item Or non solvable, if $p$ does not divide $m$.
\end{itemize}
\end{quote}
 
 Every non-solvable group of diffeomorphisms is rigid (\cite{CM}), so in this case, formal-analytic rigidity automatically arises from the nature of the holonomy group. In the general case, let us observe that the group is embeddable in a holomorphic flow. Then, if $H_1$ and $H_2$ are two such groups that are formally conjugate, they are necessarily analytically conjugate. As a result, there is always a formal-analytic rigidity phenomenon. In the abelian case, this does not follow from the nature of the group: not every finitely generated abelian group of diffeomorphisms is rigid. If $G < \diff(\CC, 0)$, $G$ is rigid if and only if $G_0 = G \cap \diff_0(\CC, 0)$ is not cyclic. However, $H_0 = H \cap \diff_0(\CC, 0)$ is a group of diffeomorphisms of the form $\exp(n \cdot \XX)$, where $n$ belongs to a subgroup of $(\ZZ, +)$, which is necessarily cyclic.

Rigidity is a phenomenon that frequently occurs in the study of nilpotent singularities. In fact, if $G$ and $G'$ are two formally equivalent subgroups of $\diff(\CC, 0)$ that are either finite, abelian non-exceptional, solvable non-exceptional, or non-solvable, they are analytically equivalent. Here, "exceptional" means that the associated subgroups of diffeomorphisms tangent to the identity are not cyclic \cite{CM}. As a result, in the generalized cusp case, the formal-analytic moduli are generically trivial. However, in the generalized Poincaré-Dulac case, rigidity does not always stem from the rigidity of the groups themselves but, as mentioned earlier, from the properties of the foliations (rigidity of Poincaré-Dulac singularities).

Let us relate these results to those of Str\'{o}\.{z}yna and \.{Z}o{\l}{\c{a}}dek in \cite{SZ2}. In that paper the authors consider nilpotent vector fields, which they call Bogdanov-Takens singularities. Using Takens' normal form, they assume that the vector fields have the form $\mathbf{X}=\mathbf{V_H}+ \mathbf{W} $, where
$$
\mathbf{V_H}= (y+(\lambda+1)x^r)\pa{x} - \lambda r x^{2r-1} \pa{y}
$$
is quasi-homogeneous with degrees $\deg (x)=1$, $\deg \left( \pa{x} \right)=-1$, $\deg (y)=r$, $\deg \left( \pa{y} \right) = -r$, and $\mathbf{W}$ gathers the (quasihomogeneous) higher  order terms. In their notation, $r$ can assume non-integer values. When $r=p\in \NN$, we are in the generalized saddle case. Indeed, using 1-forms, the foliation is generated by
\begin{align*}
\omega & = (y+(\lambda +1)x^p) dy + \lambda p x^{2p-1} dx + h.o.t. \\
& = \dfrac{1}{2} d(y^2+\lambda x^{2p}) + ((\lambda +1) x^p + h.o.t.) dy,
\end{align*}
or equivalently,
$$
\omega'=  2\omega = d(y^2+\lambda x^{2p}) + (2(\lambda +1) x^p + h.o.t.) dy.
$$
A linear change of variables $(x,y)\rightarrow (ax,y)$, with $\lambda a^{2p}=1$ leads to
$$
d(y^2+ x^{2p}) + (2(\lambda +1)\lambda^{-1/2} x^p + h.o.t.) dy, 
$$
so, in the notation of \eqref{formaprenormal}, $\alpha= 2(\lambda+1) \lambda^{-1/2}$. By \eqref{condicion_gPD}, we must have an equality
$$
2(\lambda+1) \lambda^{-1/2}= - \frac{2(m+2p)}{\sqrt{p(m+p)}},
$$
for some $m\in \NN$. This equation has two solutions, namely
$$
\lambda= \frac{p}{m+p} \text{ and } \lambda= \frac{m+p}{p}.
$$
In \cite[Section 5.1]{SZ2}, it is assumed that $\lambda \geq 1$, so we set $\lambda = 1 + \dfrac{m}{p}$. The case $\lambda = k \in \NN$ is exactly when $p \mid m$, which is the situation where the holonomy group is abelian and rigidity is not guaranteed by the nature of the group. In the research of \cite{SZ2}, this is the case where the normal form is more complicated, requiring substantially greater effort to obtain. It would be interesting to establish a complete relationship between the cases and situations studied by Str\'{o}\.{z}yna and \.{Z}o{\l}{\c{a}}dek in their works, and the results concerning projective holonomy studied by Cerveau, Moussu, Meziani, Loray, Sad, and others, but this is beyond the scope of our research here.

 \subsection*{Acknowledgments}
 The second author wants to thank the Pontificia Universidad Católica del Perú (Lima, Peru) for hosting him several times while preparing this paper, and for the wonderful ambiance of the Mathematics section. 
 
 Both authors want to thank the referees of the manuscript for their valuable and constructive suggestions which have helped to improve both the exposition and the contents.
 
 \subsection*{Funding statement} First author was partially supported by the Pontificia Universidad Católica del Perú project VRI-DGI 2019-01-0021.
 
Second author partially supported by Ministerio de Ciencia e Innovación of Spain, under Project PID2019-105621GB-I00 (Javier Sanz Gil and Fernando Sanz Sánchez, coords.). This work is part of the Project PID2022-139631NB-I00, funded by MICIU/AEI /10.13039/501100011033 and by FEDER, UE.

 \subsection*{Conflict of interest}
 Both authors declare that there is no potential conflict of interest with respect to the research,
author-ship, and publication of this article.


\begin{thebibliography}{AAA99}

\bibitem[{\bf BMeS}]{BMS}
{\sc M. Berthier, R. Meziani and P. Sad}. {\it On the
classification of nilpotent singularities}, Bull. Sci. Math. 123
(1999) p. 351-370.


\bibitem[{\bf CaS1}]{Camacho-Sad} {\sc C. Camacho, P. Sad.} \textit{Invariant varieties through singularities of holomorphic vector fields}. Annals of Mathematics vol. 115, no. 3 (1982), 579--595.

\bibitem[{\bf CaS2}]{CS2} \textsc{C. Camacho, P. Sad}. \textit{Pontos singulares de equaç\~{o}es diferenciais analiticas}. 16 Colóquio Brasileiro de Matemática. IMPA (1987).



%
%
%
%
%
%

\bibitem[{\bf CeMo}]{CM}
{\sc D. Cerveau, R. Moussu}. {\it Groupes d'automorphismes de
$(\CC,0)$ et \'equations diff\'erentielles $ydy+\cdots=0$}, Bull.
Soc. Math. France \textbf{116} (1988) 459-488.

%
%
%

\bibitem[{\bf D}]{Dulac} \textsc{H. Dulac}. \textit{Solutions d'un système d'équations différentielles dans le voisinage de valeurs singulières}. Bull. Soc. Math. France \textbf{40} (1912), 324--383.



\bibitem[{\bf FMN}]{FMN3} \textsc{P. Fernández Sánchez, J. Mozo Fernández, H. Neciosup}. \textit{Dicritical Nilpotent Holomorphic Foliations}. Disc. and Cont. Dyn. Systems vol. \textbf{38}, no. \textbf{7} (2018), 3223--3237.

\bibitem[{\bf IY}]{IY} \textsc{Y. Ilyashenko, S. Yakovenko}. \textit{Lectures on Analytic Differential Equations.} Graduate Studies in Mathematics \textbf{86}, AMS (2008).

%


\bibitem[{\bf L}]{Lins} \textsc{A. Lins Neto}. \textit{Construction of singular holomorphic vector fields and foliations in dimension two}. J. Differential Geometry \textbf{26} (1987), 1--31.

\bibitem[{\bf LSc}]{LNS} \textsc{A. Lins Neto, B. Scárdua}. \textit{Folheações Algébricas Complexas}.  Projeto Euclides. IMPA (2015).

\bibitem[{\bf Lo1}]{Loray} \textsc{F. Loray}. \textit{Pseudo-groupe d'une singularité de feuilletage holomorphe en dimension deux}. Notes of 2005, published later in: Ensaios Math. 36. Sociedade Brasileira de Matemática, Rio de Janeiro (2021).

\bibitem[{\bf Lo2}]{Loray1} {\sc F. Loray}. {\it R\'{e}duction formelle des singularit\'{e}s cuspidales de champs de vecteurs analytiques}. J. of Diff. Equations \textbf{158} (1999), 152--173.

\bibitem[{\bf MR}]{MR2} \textsc{J. Martinet, J.-P. Ramis.} \textit{Classification analytique des \'{e}quations diff\'{e}rentielles
non lin\'{e}aires r\'{e}sonnantes du premier ordre}. Ann. Sci. \'{E}cole Normale Sup.
\textbf{16} (1983) 571--621.
\bibitem[{\bf MaMo}]{MM}
{\sc J.F. Mattei, R. Moussu}, {\it Holonomie et int\'egrales
premi\'eres}, Ann. Sci. \'Ecole Normale Sup. \textbf{13} (1980) 469--523.

\bibitem[{\bf MaSa}]{Mattei-Salem} \textsc{J.-F. Mattei, E. Salem}. \textit{Modules formels locaux de feuilletages holomorphes}. 	Prepublication 274, Laboratoire E. Picard, Toulouse (France). 	\url{arXiv:math/0402256} [math.DS]. 

\bibitem[{\bf Me1}]{Meziani-tesis}

{\sc R. Meziani}, {\it Probl\`{e}me de modules pour des
\'{e}quations di{f}{f}\'{e}rentielles d\'{e}g\'{e}n\'{e}r\'{e}es
de $ (\CC^2, 0) $}, PhD. thesis, Univ. Rennes I (1992)

\bibitem[{\bf Me2}]{Meziani} \textsc{R. Meziani.}
\textit{Classification analytique d'\'{e}quations diff\'{e}rentielles
$ydy+\cdots =0$ et espaces de modules}. Bol. Soc. Brasil Mat.
\textbf{27} (1996), 23--53.

\bibitem[{\bf MeS}]{MS} \textsc{R. Meziani, P. Sad.}
\textit{Singularit\'{e}s nilpotentes et int\'{e}grales premi\`{e}res}. Publ.
Mat. \textbf{51} (2007), 143--161.

\bibitem[{\bf Mo}]{Moussu} \textsc{R. Moussu.}
\textit{Holonomie \'{e}vanescente des \'{e}quations diff\'{e}rentielles d\'{e}g\'{e}n\'{e}r\'{e}es transverses,} in \textit{Singularities and Dynamical systems}. North-Holland (1985), 151--173.

\bibitem[{\bf P}]{poincare} \textsc{H. Poincar\'{e}}. \textit{Pr\'{e}mi\`{e}re Th\`{e}se.- Sur les propri\'{e}t\'{e}s des fonctions d\'{e}finies par les \'{e}quations aux diff\'{e}rences partielles}, en \textit{Th\'{e}ses pr\'{e}sent\'{e}es \`{a} la Facult\'{e} des Sciences de Paris pour obtenir le Grade de Docteur \`{e}s Sciences Math\'{e}matiques}. Gauthier-Villars (1879).



\bibitem[{\bf Sc}]{Scardua} \textsc{B. Scárdua}. \textit{Holomorphic Foliations with Singularities. Key Concepts and Modern Results}. Springer Latin American Mathematical Series (2021).

\bibitem[{\bf Se}]{Seidenberg} \textsc{A. Seidenberg}. \textit{Reduction of Singularities of the Differential Equation $Ady=Bdx$}. Amer. J. of Math. (1968), 248--269.

\bibitem[{\bf St}]{Strozyna} \textsc{E. Str\'{o}\.{z}yna.} \textit{The
analytic and formal normal form for the nilpotent singularity.
The case of generalized saddle-node}. Bull. Sci. Math.
\textbf{126} (2002), 555--579.

\bibitem[{\bf StZ1}]{SZ} \textsc{E. Str\'{o}\.{z}yna, H. \.{Z}o{\l}{\c{a}}dek}. \textit{The Analytic and Normal Form for the Nilpotent Singularity}. J. Diff. Equations \textbf{179} (2002), 479--537.

\bibitem[{\bf StZ2}]{SZ2} \textsc{E. Str\'{o}\.{z}yna, H. \.{Z}o{\l}{\c{a}}dek}. \textit{Analytic properties of the complete formal normal form for the Bogdanov-Takens singularity}. Nonlinearity \textbf{34} (2021), 3046--3082.

\bibitem[{\bf T}]{Takens} \textsc{F. Takens.} \textit{Singularities of vector fields}. Inst. Hautes \'{E}tudes Sci. Publ. Math. \textbf{43} (1974) 47--100.

\bibitem[{\bf V}]{vdE} \textsc{A. van den Essen.} \textit{Reduction of Singularities of the Differential Equation $Ady=Bdx$}, en \textit{\'{E}quations Diff\'{e}rentielles et Syst\`{e}mes de Pfaff dans le Champ Complexe} (R. G\'{e}rard y J.-P. Ramis, eds.). Springer LNM \textbf{712} (1979), 44--59.

\bibitem[{\bf Z}]{Zoladek} \textsc{H. \.{Z}o{\l}{\c{a}}dek}. \textit{The Monodromy Group}. Birkhäuser (2006).

\end{thebibliography}
\end{document}